\documentclass[12pt]{amsart}
\usepackage[utf8]{inputenc}
\usepackage{setspace}
\usepackage[english]{babel}
\usepackage{mathabx}
\usepackage{framed}
\usepackage{soul}
\usepackage[all]{xy}
\usepackage[margin=1in]{geometry}
\usepackage{arydshln}
\usepackage[colorinlistoftodos]{todonotes}
\usepackage{pb-diagram}
\usepackage[mathscr]{euscript}
\usepackage{multicol}
\usepackage{stmaryrd} 
\usepackage{empheq} 
\usepackage{tikz-cd}
\tikzset{
  symbol/.style={
    draw=none,
    every to/.append style={
      edge node={node [sloped, allow upside down, auto=false]{$#1$}}}
  }
}

\usepackage{amsmath,amsfonts,amssymb,amsthm,mathrsfs,latexsym,mathtools}
\usepackage{commath}
\usepackage[T1]{fontenc}
\usepackage{hyperref}
\usepackage{caption} 
\usepackage{subcaption} 

\usepackage{lipsum}

\DeclareMathAlphabet{\mathbbmsl}{U}{bbm}{m}{sl}

\usetikzlibrary{matrix}
\title{Elliptic Surfaces}

\newcommand{\C}{\mathbb{C}} 
\newcommand{\Z}{\mathbb{Z}}

\newcommand{\bbQ}{\mathbb{Q}}

\newtheorem{theorem}{Theorem}[section]
\newtheorem{definition}[theorem]{Definition}
\newtheorem{proposition}[theorem]{Proposition}
\newtheorem{corollary}[theorem]{Corollary}

\newtheorem{lemma}[theorem]{Lemma}
\newtheorem{conjecture}[theorem]{Conjecture}

\newtheorem{claim}[theorem]{Claim}

\newtheorem{thm}{Theorem}

\title{Monodromy of Primitive Vanishing Cycles for Hypersurfaces in $\mathbb P^4$}

\begin{document}
\author{Yilong Zhang}
\date{Jan 13, 2024}
\maketitle
\begin{abstract}
    Let $X$ be a smooth projective variety. Schnell showed that the middle-dimensional primitive cohomology of $X$ is generated by tube classes, which arise from the monodromy of the vanishing homology on hyperplane sections. Clemens asks if the theorem is still true when we restrict the generating set to the tube classes over the class of a single vanishing sphere of nodal degeneration. We prove this is true for hypersurfaces in $\mathbb P^4$. The proof is based on the degeneration of a hypersurface to the union of hypersurfaces of lower degrees. 
\end{abstract}

\tableofcontents

\section*{Introduction}
Let $X\subset \mathbb P^N$ be a smooth projective variety over $\mathbb C$ of dimension $n$. Let $X_H=X\cap H$ be a general hyperplane section. Then, the Lefschetz hyperplane theorem says the homology of $X$ is captured by $X_H$ up to degree $n-1$. The question is, can $n$-th homology of $X$ be recovered from hyperplane sections?

There is a construction called tube mapping that recovers $n$-th primitive homology on $X$ by considering the monodromy of vanishing cycles on $X_H$, as $H$ moves in the space of all smooth hyperplane sections: Let $\delta\in H_{n-1}(X_H,\Z)_{van}$ be a class in the vanishing homology. Let $l\in \pi_1(\mathbb O^{sm},H)$ be a loop in the universal family of smooth hyperplane sections of $X$. Then the trace of $\delta$ along $l$ defines a relative class in $H_n(X,X_H,\Z)$. When $l$ stabilizes $\delta$ by monodromy, then the resulting class $\tau_l(\delta)$ lives in the primitive homology $H_n(X,\Z)_{prim}$. This resulting map 
\begin{equation}\label{Intro_TubeMap}
    \bigoplus_{g\in \pi_1(\mathbb O^{sm},H)} H_{n-1}(X_H,\Z)_{van}^g\to H_n(X,\Z)_{prim}
\end{equation}
$$(g,\delta)\mapsto \tau_g(\delta)$$
is called tube mapping. Assume that the vanishing homology is nonzero, Schnell proved the following.
\begin{theorem} (Schnell, \cite{Schnell})
    The tube mapping \eqref{Intro_TubeMap} has a cofinite image.
\end{theorem}

In other words, over rational numbers, any primitive class on $X$ can be obtained as a finite sum of "tubes" plied up by vanishing cycles on a hyperplane section.


When $n=\dim(X)$ is an odd number, the middle-dimensional primitive cohomology is a pure Hodge structure and defines a complex torus $J_{prim}=F^{\frac{n+1}{2}}H^n(X,\C))_{prim}/H_n(X,\Z)_{prim}$. By varying the hyperplane sections, the vanishing (co)homology forms a local system over $\mathbb O^{sm}$. So the underlying \'etale space $T$ of the local system is naturally an analytic covering space of $\mathbb O^{sm}$. There is a real-analytic map, 
$$\Phi: T\to J_{prim},$$
called the \textit{topological Abel-Jacobi map} (c.f. \cite[p.3]{Zhang_TAJ}, \cite{Zhao}), 
It generalizes Griffiths' Abel-Jacobi map to topological cycles.

It turns out that the tube mapping is simply the associated $\pi_1$ of the topological Abel-Jacobi map — Any loop $l$ in $\pi_1(\mathbb O^{sm},H)$ lifts to $\pi_1(T,\delta)$ for $\delta\in H_{n-1}(X_H,\Z)$ if and only if $l_*\delta=\delta$ (cf. \cite[Lemma 6.2]{Zhang_Cubic3folds}). So Schnell's theorem implies that the map $\Phi$ captures enough topological information in the middle dimensional cohomology of $X$.

As a consequence of Schnell's theorem, there is at least one connected component of $T$ that realizes a cofinite image of the tube mapping. The question is which component does the job?

Among all the components of $T$, there is a distinguished component $T_{v}$ that consists of \textit{primitive vanishing cycles}. Equivalently, there is a unique component that contains a vanishing cycle of nodal degeneration. These classes consist of classes of vanishing spheres in a hyperplane section near a hyperplane section $X_{H_0}$ that has an ordinary node.

\begin{conjecture}(Clemens)\label{Conj_Clemens}
The restriction of tube mapping on primitive vanishing cycles has a cofinite image.
\end{conjecture}

It is verified in the case when $\dim(X)=1$ \cite{Fu} and when $X$ is a cubic threefold \cite{Zhang_Cubic3folds}.

We will prove the conjecture for hypersurfaces in $\mathbb P^4$.

\begin{theorem}\label{MainThm}
    Let $X$ be a smooth hypersurface of $\mathbb P^4$ of degree at least 3. Let $\delta\in H_2(X_H,\Z)$ be a primitive vanishing cycle, then the tube mapping over $\delta$
    \begin{equation}\label{eqn_tubemappingPVC}
       \Phi_{*}^v: G=\{g\in \pi_1(\mathbb O^{sm},*)|g_*\delta=\delta\}\to H_3(X,\Z)_{prim},\ g\mapsto \tau_g(\delta)
    \end{equation}
    has a cofinite image.
\end{theorem}

Our strategy is to first reduce the proof of Theorem \ref{MainThm} to prove the image is nonzero. This essentially follows from the monodromy conjugacy property of primitive vanishing cycles and Picard-Lefschetz theory.

When $d=3$,  a primitive vanishing cycle is the difference of the classes of two disjoint lines $[L_1]-[L_2]$ on a cubic surface. Then the result follows from that the (topological) Abel-Jacobi map on cubic threefolds is holomorphic and nonzero (cf. Proposition \ref{prop_cubic3fold}). 

For $d\ge 4$, there is no obvious geometric description for primitive vanishing cycles, and the topological Abel-Jacobi map is no longer holomorphic, so it is not clear how to show tube mapping \eqref{eqn_tubemappingPVC} is nonzero directly. Instead, we inductively degenerate $X$ to the union of hypersurface $Y$ of degree $d-1$ and a hyperplane $P$, meeting transversely. To prove the tube mapping on $X$ is nonzero, we choose a uniform Lefschetz pencil in the family and try to push a nonzero tube class on $Y$ away from $Y\cap P$. Then we show it can deform to the nearby fiber. Hodge theory of degeneration guarantees that this class is nonzero.

\subsection{Higher Dimensions}
The degeneration strategy can be applied to higher dimensional hypersurfaces. However, the question will be how to describe the primitive vanishing cycles on hyperplane sections of cubic $2n-1$-fold. For example, the Abel-Jacobi map of planes on cubic 5-fold \cite{Collino} is similar to the cubic threefold case and allows us to deduce the tube mapping associated with the difference of two disjoint planes $[P_1]-[P_2]$ on a cubic fourfold is nonzero. However, $[P_1]-[P_2]$ has self-intersection 6, so it is not a primitive vanishing cycle (nor a multiple of a primitive vanishing cycle). So, to answer Conjecture \ref{Conj_Clemens}, the question is still to describe the geometry of primitive vanishing cycles and understand their (topological) Abel-Jacobi image for cubic hypersurfaces in $\mathbb P^{2n}$.\\

\noindent\textbf{Acknowledgement.} This paper is derived from the author's thesis. The author extends sincere gratitude to his advisor, Herb Clemens, for introducing him to this topic, engaging in helpful discussions over the years, and providing constant encouragement. The author would like to thank Arun Debray and Laurentiu Maxim for helpful conversations and for inviting the author to give seminar talks on the subject. Additionally, the author expresses appreciation to Gael Meigniez and Philip Tosteson for showing him the proof of Lemma \ref{appendix}.

\section{Primitive Vanishing Cycles}
Let $X\subseteq \mathbb P^N$ be a smooth projective variety of dimension $n$. Let $\mathbb O^{sm}$ be the open subspace of $\mathbb O:=(\mathbb P^N)^*$ parameterizing smooth hyperplane sections of $X$. Let $H\in \mathbb O^{sm}$ and $i:X_H\to X$ the inclusion of the hyperplane section, then the \textit{vanishing homology} of $X_H$ is 
\begin{equation}\label{eqn_vanishinghomology}
  H_{n-1}(X_H,\Z)_{\textup{van}}:=\ker \big (i_*:H_{n-1}(X_H,\Z)\to H_{n-1}(X,\Z)\big ).  
\end{equation}

Now let $H$ deform and get close to a hyperplane $H_0$ that is tangent to $X$, there is a topological $(n-1)$-sphere, called the vanishing sphere, on the nearby fiber $X_{H_t}$, such that the sphere $S^{n-1}$ goes to the node on $X_{H_0}$ when $t\to 0$.

\begin{definition}\normalfont
    Let $H$ be a general hyperplane, then we call $\delta\in H_{n-1}(X_H,\Z)$ a primitive vanishing cycle if it is monodromy conjugate to the class of a vanishing sphere. In other words, there exists a path $l\subseteq \mathbb O^{sm}$ joining $H_t$ to $H$ and the image $l_*([S^{n-1}])=\delta$ under the trivialization $H_t\cong H$ along the path.
\end{definition}

One also refers to \cite[Appendix A]{Zhang_Cubic3folds} for an equivalent description using vanishing cohomology.

\begin{proposition} \label{Prop_PVCconjugate}
The set of primitive vanishing cycles are in the same orbit of the monodromy action
\begin{equation}\label{eqn_GlobalMonodromy-General}
    \rho_{\textup{van}}:\pi_1(\mathbb O^{sm},H)\to \textup{Aut}H_{n-1}(X_H,\Z)_{van}.
\end{equation}

\end{proposition}
\begin{proof}
This follows from the fact that the dual variety $X^*$ is irreducible and all vanishing cycles of nodal degenerations $\{X_t\}_{t\in \Delta}$ obtained from above are conjugate to each other \cite[Proposition 3.23]{Voisin}.
\end{proof}

\begin{proposition}\label{Prop_PVC-generation}
The set of all primitive vanishing cycles in $H_{n-1}(X_H,\Z)_{van}$ generates a sublattice of full rank.
\end{proposition}
\begin{proof}
The primitive vanishing cycles contain the classes of vanishing spheres through a Lefschetz pencil as a subset. The claim follows that the latter generates the vanishing homology over $\bbQ$ \cite[Lemma 2.26]{Voisin}.
\end{proof}

\section{Properties of Tube Mapping}
In this section, we summarize some properties of tube mapping. Note the key Lemma \ref{lemma} only holds for primitive vanishing cycles.
\subsection{Lefschetz Pencil} \label{sec_LefPencil}
Let $X\subseteq \mathbb P^N$ be a smooth projective variety. Then there is a line $\mathbb L\subseteq \mathbb O$, called Lefschetz pencil \cite[Section 2.1]{Voisin}, such that

\begin{itemize}
    \item[(1)] For general $t\in \mathbb L$, the hyperplane section $X_{H_t}$ is smooth.
    \item[(2)] There are at most finitely many $t_{1},\ldots,t_{d^*}\in \mathbb L$ such that $X_{H_{t_i}}$ is singular. Moreover, the singularity is a single ordinary node.
\end{itemize}

$d^*=\deg(X^*)$ is the degree of the dual variety. 

Let $U_{X}$ be the open subspace by removing the points $t_i$ from the pencil. Zariski showed that 

\begin{lemma} \cite[Theorem 3.22]{Voisin}
    $\pi_1(U_X,*)\to \pi_1(\mathbb O^{sm},*)$ is surjective.
\end{lemma}

Consequently, the tube mapping \eqref{eqn_tubemappingPVC} can be expressed as tubes along loops contained in a Lefschetz pencil.

\subsection{Tube Mapping for a Smooth Family}
The construction of tube mapping is topological, so it should be unchanged when the manifold deforms, as long as the topology of the hyperplane sections in a Lefschetz pencil is unchanged.

More precisely, let $\{X_s\}_{s\in \Delta}$ be a deformation of a smooth projective variety $X=X_0$ in $\mathbb P^N$, then by Ehresmann's theorem the cohomology $H_n(X)\cong H_n(X_s)$ by the trivialization $X\cong X_t$. So we can compare the tube mapping image for $s\in \Delta$. In fact, this is unchanged.

\begin{proposition}\label{prop_localconst}
    The image of the tube mapping on a primitive vanishing cycle is constant in the family $\{X_s\}_{s\in \Delta}$. In particular, the image of tube mapping is invariant under smooth deformation.
\end{proposition}

\begin{proof}
First, up to shrink the disk, we can choose a \textit{uniform} Lefschetz pencil for the family $\{X_s\}_{s\in \Delta}$ in the sense that there is a line $\mathbb L\subseteq (\mathbb P^N)^*$ such that property (1) and (2) are satisfied for all $s\in \Delta$. This is obvious because the Lefschetz pencil $\mathbb L$ for $X_0$ intersects the dual variety $X_0^*$ transversely at smooth points. Then $\mathbb L$ is transverse to $X_s^*$ at distinct points $t_{1,s},\ldots, t_{d^*,s}$ for $|s|$ small and is automatically a Lefschetz pencil for $X_s$.

Next, up to further shrink the disk $\Delta$ and remove suitably small closed neighborhood of $t_{1,0},\ldots, t_{d^*,0}$ in $\mathbb L$, there is an analytic open subset $U_{\mathbb L}\subseteq \mathbb L$ such that 

\begin{itemize}
    \item[(3)] For all $s\in \Delta$ and all $t\in U_{\mathbb L}$, $X_s\cap H_t$ is smooth. 
\end{itemize}

Finally, we choose a common base point $t_0\in U_{\mathbb L}$, and a primitive vanishing cycle $\delta_0\in H_{n-1}(X_s\cap H_{t_0},\Z)$ for all $t\in \Delta$. Then the subgroup $G_s\subseteq \pi_1(U_{\mathbb L},t_0)$ that stabilizes the primitive vanishing cycle $\delta_0$ in the hyperplane sections $\{X_s\cap H_t\}_{t\in U_{\mathbb L}}$ are the same $G$ since the topology of the family is the same. Therefore, the tube mapping for $X_s$ 
$$G\to H_{n}(X_s,\Z)_{prim}\cong H_n(X,\Z)_{prim}$$
is constant for the family $\{X_s\}_{s\in \Delta}$.
\end{proof}

\subsection{Monodromy Conjugation}
The tube mapping associated with a primitive vanishing cycle $\delta$ is to take the subgroup $G\subseteq \pi_1(U_X,*)$ which stabilizes $\delta$ and consider the homomorphism
\begin{equation}\label{eqn_tubemap}
    G\to H_n(X,\Z)_{prim}. 
\end{equation}

\begin{proposition}\label{prop_conjTube}
    The image of \eqref{eqn_tubemap} is independent of the choice of $\delta$.
\end{proposition}
\begin{proof}
By Proposition \ref{Prop_PVCconjugate}, for any two primitive vanishing cycle $\delta$ and $\delta'$, there is a loop $l\in \pi_1(U_X,*)$ such that $l_*\delta'=\delta$. So if $G$ stabilizes $\delta$, $G'=l^{-1}.G.l$ stabilizes $\delta'$.

Now we need to show 
\begin{equation}\label{eqn_conjTube}
   \tau_g(\delta)=\tau_{l^{-1}.g.l}(\delta').
\end{equation}

We can express $\tau_{l^{-1}.g.l}(\delta')$ as sum of three relative classes
$$\tau_{l^{-1}}(g_*l_*(\delta'))+\tau_{g}(l_*(\delta'))+\tau_{l}(\delta')=\tau_{l^{-1}}(\delta)+\tau_{g}(\delta)+\tau_{l}(\delta')$$
in $H_n(X,X_H,\Z)$. Since $\delta$ and $\delta'$ are at the two ends of $l$, by definition of tube mapping, $\tau_{l^{-1}}(\delta)=-\tau_{l}(\delta')$, so the first and the last term cancels out, which establishes the equality \eqref{eqn_conjTube}.
\end{proof}

\subsection{A Key Lemma} From now on, we assume $X$ is a smooth hypersurface. The following lemma is based on the various properties of tube mapping and primitive vanishing cycles mentioned above. 
\begin{lemma}\label{lemma}
Suppose $X\subset \mathbb P^{n+1}$ is a smooth hypersurface. Assume the tube map \eqref{eqn_tubemap} is nonzero, then the image of the tube map is cofinite.
\end{lemma}

\begin{proof}
We can choose a smooth hypersurface $W\subset\mathbb P^{n+2}$ containing $X$ as a smooth hyperplane section. Choose a general pencil $\mathbb L_W$ of hyperplane sections of $W$ passing through $X=W\cap H_{s_0}$ and let $U_W\subseteq \mathbb L_W$ be the open subspace parametrizing smooth hyperplane sections in the pencil. Then there is a monodromy action
\begin{equation}\label{eqn_Mon4fold}
   \rho:\pi_1(U_W,s_0)\to \textup{Aut}H_{n}(X,\mathbb Q)_{prim}. 
\end{equation}

Lefschetz hyperplane theorem implies the pushforward $H_{n}(W)\to H_{n}(\mathbb P^{n+2})$ is an isomorphism, so the primitive homology $H_{n}(W)_{prim}$ is zero. In particular the vanishing homology group $H_{n}(X,\bbQ)_{van}$ coincides with $H_{n}(X,\bbQ)_{prim}$. By a classical result \cite[Theorem 3.27]{Voisin}, the representation $\rho$ is \textit{irreducible}. 

On the other hand, by Proposition \ref{prop_localconst}, the image of the tube mapping $\textup{Im}(\Phi_{*}^v)$ is locally constant, so it makes sense to talk about the monodromy of $\textup{Im}(\Phi_{*}^v)$ under the action of $\rho$. This action is invariant as a consequence of Proposition \ref{prop_conjTube}.

Consequently, by tensoring with $\bbQ$, $\textup{Im}(\Phi_{*}^v)\otimes \mathbb Q$ is a $\rho$-sub-representation of $H_{n}(X,\bbQ)_{prim}$. Now, the irreduciblity of $\rho$ together with our assumption implies that $\textup{Im}(\Phi_{*}^v)\otimes \mathbb Q$ has to be the whole $H_{n}(X,\mathbb Q)_{prim}$. Therefore, $\textup{Im}(\Phi_{*}^v)\subseteq H_{n}(X,\mathbb Z)_{\textup{prim}}$ is cofinite. 

\end{proof}

\section{Cubic Threefolds}
In this section, we will give a proof of Theorem \ref{MainThm} for the case when $d=3$ using Lemma \ref{lemma}. In fact, a stronger result is proved in \cite[Proposition 6.2]{Zhang_Cubic3folds} — the tube mapping on primitive vanishing cycle \eqref{eqn_tubemappingPVC} is subjective over $\Z$. The proof is based on relating a compactification of $T_v$ to the theta divisor of the intermediate Jacobian of $X$.

\begin{proposition}\label{prop_cubic3fold}
    Tube mapping for a smooth cubic threefold is nonzero and, therefore, has a cofinite image.
\end{proposition}
\begin{proof}
For cubic threefold, the covering map $\pi:T_v\to \mathbb O^{sm}$ is finite of degree 72. Take a Lefschetz pencil $\mathbb L$ and let $U=\mathbb L\cap \mathbb O^{sm}$. Then the fiber product
\begin{equation}\label{eqn_cubicTU}
   C_U:=T_v\times_{\mathbb O^{sm}} U\to U 
\end{equation}
is a connected covering space between Riemann surfaces. 

It extends to a branched cover between compact Riemann surfaces
$$C\to \mathbb L\cong \mathbb P^1.$$

On the other hand, the topological Abel-Jacobi map $\Phi:C_U\to J(X)$ is holomorphic. One way to see holomorphicity is that $\Phi$ can be expressed as $\int_{L_{2,t}}^{L_{1,t}}$ in a branch of $C_U$, where $L_{1,t}$ and $L_{2,t}$ are disjoint lines on cubic surface $X_{H_t}$. Since $C$ is smooth with complex dimension one, $\Phi$ extends to a holomorphic map
 $\bar{\Phi}:C\to J(X)$. Then clearly, its induced map on the fundamental group is nonzero. Otherwise, $\bar{\Phi}$ lifts to the universal cover $C\to \mathbb C^5$ and has to be constant, which contradicts the nontriviality of the Abel-Jacobi map.

Finally, the tube map on Lefschetz pencil factors through $\pi_1(C_U,*)\twoheadrightarrow \pi_1(C,*)\to \pi_1(J(X),*)$, so the image is nonzero. Now the results follow from Lemma \ref{lemma}.
\end{proof}

\section{Lefschetz Pencil in a Family}\label{sec_LefschetzInFamily}
In this section, we will provide some results in preparation for the next section. 
\subsection{Degeneration of Dual Varieties}

Consider a one-parameter family of hypersurfaces of degree $d$ degenerate to the union of two hypersurfaces of lower degrees. Let $F_d$, $F_{d_1}$, and $F_{d_2}$ be the defining equation of a degree $d$ hypersurface, and let $F_{d_1}$ and $F_{d_2}$ be the defining equation of a degree $d, d_1$, and $d_2$ hypersurfaces in $\mathbb P^{n+1}$, respectively. Assume the three equations are general. Consider the one-parameter family defined by 
\begin{equation}\label{eqn_degdual}
    F_s:=sF_d+F_{d_1}F_{d_2}.
\end{equation}

Then, for each $s\neq 0$, the dual variety of $X_s=\{F_s=0\}$ is an irreducible hypersurface of degree $d^*=d(d-1)^n$ in the dual space. The question is, what is the dual variety when $t=0$? Of course the dual variety of $F_{d_1}=0$ and $F_{d_2}=0$ are components, but what else? We give an answer in \cite{Zhang} by taking the closure of $\{X_s^*\}_{s\in \Delta^*}$ inside $(\mathbb P^{n+1})^*\times \Delta$ and define   $X_0^*$ as the scheme-theoretic fiber over $0$. $X_0^*$ is called the dual variety in the limit associated with the family \eqref{eqn_degdual}.

\begin{theorem}\cite{Zhang}
    $\{X_s^*\}_{s\in \Delta^*}$ is reducible and consists of four components: $X_{d_1}^*$, $X_{d_2}^*$, $(X_{d_1}\cap X_{d_2})^*$ (with multiplicity two) and $(X_{d}\cap X_{d_1}\cap X_{d_2})^*$.   
\end{theorem}

Now we have a variant of uniform Lefschetz pencil as in the proof of Proposition \ref{prop_localconst}

\begin{lemma}\label{RobustLemma}
Up to shrinking $\Delta$ to a smaller disk centered at $s=0$, we can choose a line $\mathbb L\subseteq (\mathbb P^{n+1})^*$, such that $\mathbb L$ is transverse to all $X_s^*$ for $s\in \Delta$. 
\end{lemma}
\begin{proof}
This argument is based on continuity. First, we choose $\mathbb L$ to be transverse to $X^*_0$, i.e., $\mathbb L$ is disjoint from the singular locus of $X^*_0$. Then, since $X_0^*$ is the limit of $X_s^*$, we can take a small analytic neighborhood $\mathcal{U}$ of $\textup{Sing}(X_0^*)$ in $(\mathbb P^{n+1})^*$ such that $\mathcal{U}$ contains $\textup{Sing}(X_s^*)$ for all $s\in\Delta$ up to restricting to a smaller disk. Therefore, $\mathbb L$ intersect $X_s^*$ along the smooth locus $(X_s^*)^{sm}$ for each $s\in \Delta$.
\end{proof}

\begin{corollary}\label{cor_uniformLefschetz}
There exists an \textup{analytic} open subset $U\subseteq \mathbb L$ obtained by removing finitely many closed disks from $\mathbb L$, such that
\begin{itemize}
    \item[(i)] for all $s\in \Delta^*$ and $t\in U$, $H_t\cap X_s$ is smooth, and
    \item[(ii)] for each $t\in U$, $H_t\cap X_{d_i}$ and $H_t\cap X_{d_1}\cap X_{d_2}$ are smooth.
\end{itemize}
\end{corollary}

\subsection{A Family of Local Vanishing Cycles}

Recall that in the previous section, we associated the family of hypersurfaces with a family of dual varieties 
$$f:\bigcup_{s\in\Delta}X_s^*\to \Delta$$
with $X_0^*$ as the dual variety in the limit, which contains $X_{d_1}^*$ as an irreducible component (assuming $d_1\ge 2$). Also, we have uniform Lefschetz pencil $\mathbb L$ for all $s\in\Delta$. 

Choose a point $p\in \mathbb L\cap X_{d_1}^*$, so in particular, $p$ is a smooth point of $X_{d_1}^*$ and away from other components of $X_0^*$. By inverse function theorem, up to shrinking to a smaller disk, we can find $\mathcal{N}(s)\in \mathbb L$ varying differentiably with respect to $s\in \Delta$ such that $\mathcal{N}(s)\in \mathbb L\cap (X_s^*)^{sm}$ and $\mathcal{N}(0)=p$. In particular, $\mathcal{N}$ defines a $C^{\infty}$ section whose image lies in the smooth part $(X_s^*)^{sm}$ when $s\neq 0$, and additionally $\mathcal{N}(0)\in (X_{d_1}^*)^{sm}$.
\begin{figure}[ht]
    \centering
\begin{tikzcd}
\bigcup_{s\in \Delta}X_{s}^* \arrow[r,hook] \arrow{d}[swap]{f} & (\mathbb P^{n+1})^*\times \Delta \arrow[dl]\\ \Delta \arrow[bend right, right=30,swap]{u}{\mathcal{N}} \arrow[bend right=20,swap]{ur}{c}
\end{tikzcd}
\end{figure}

Note by the dual correspondence, $\mathcal{N}$ also gives the nodal locus, i.e., $\tilde{\mathcal{N}}(s)\in X_s$ is a point such that $H_{\mathcal{N}(s)}$ is tangent to $X_s$ at $\tilde{\mathcal{N}}(s)$. When $s=0$, the point $\tilde{\mathcal{N}}(0)$ lives in $X_{d_1}$. We can choose $\mathcal{N}$ so that $\tilde{\mathcal{N}}(0)$ is not on $X_d\cap X_{d_1}$.

\begin{proposition}\label{prop_c}
 Let $U$ be the analytic open subset of Lefschetz pencil in Corollary \ref{cor_uniformLefschetz}.  Choose a point $c\in U$ close to $\mathcal{N}(0)$. Then there is a family of vanishing cycles $\alpha_s$ on $X_{s}\cap H_{c}$ represented by vanishing spheres. When $s=0$, $\alpha_0$ lies in $X_{d_1}\cap H_c$. 
\end{proposition}
\begin{proof}
 
We take a small polydisk $D$ in $\mathbb P^{n+1}$ containing $\tilde{\mathcal{N}}(0)$ so $\tilde{\mathcal{N}}(s)\in D$ when $s$ is small. Also, we require $D$ to stay away from the base locus, then $D$ can thought of as living in the total space $\mathscr{X}=\{(x,s)\in \mathbb P^{n+1}\times \Delta|x\in X_s\}$. We can choose affine coordinate $x_1,...,x_{n},t$ where $t$ corresponds to the pencil $\mathbb L$. Recall that $F_s=sF_d+F_{d_1}F_{d_2}$ to be the homogeneous polynomial varying in $s$, so restriction of $F_s$ to a fixed $t$ is the equation of the hyperplane section $X_s\cap H_t$. For each $s\in \Delta$, we denote $\tilde{\mathcal{N}}(s)=(x_1^s,...,x_{n}^s,t^s)$ the nodal locus, i.e., the hyperplane section $X_s\cap H_{t^s}$ has an ordinary node at $\tilde{\mathcal{N}}(s)$. Since $\partial F_s/\partial t (\tilde{\mathcal{N}}(s))\neq 0$, the implicit function theorem implies that there is a smooth function $f_s(x_1,...,x_{n})$, such that 
$$F_s(x_1,...,x_{n},f_s(x_1,...,x_{n}))\equiv 0.$$

Moreover, $f_s$ is a holomorphic function in $x_1,...,x_{n}$ and is analytic with respect to the parameter $s$. There is a power series expansion 
$$f_s(x_1,...,x_{n})=Q_s(x_1-x_1^{s},...,x_{n}-x_{n}^s)+\textup{higher}\ \textup{powers},$$
where $Q_s$ is a nondegenerate quadric form. 

Now, by a parametric version of the holomorphic Morse lemma, there is an analytic change of coordinate $x_1',...,x_{n}'$ such that 
$$f_s(x_1',...,x_{n}')=x_1'^2+\cdots+x_{n}'^2.$$
Moreover, the change of coordinate depends analytically with respect to the parameter $s$.

Consequently, there is an analytic isomorphism
$$D\xrightarrow{\cong}\{x_1'^2+\cdots+x_{n}'^2=t\}\times \Delta$$
which preserves projection to $\Delta$.
\end{proof}

\section{Tube Class via Degeneration}

 In this section, we plan to prove Theorem \ref{MainThm} for hypersurfaces of degree at least four. Our idea is to consider degenerating a degree $d$ hypersurface $X$ into the union of a degree $d-1$ hypersurface $Y$ and a hyperplane $P$ meeting transversely. To give an example, let $X$ be a quartic 3-fold and $Y$ a cubic threefold. Let $F_X$, $F_Y$ and $F_P$ be general homogeneous polynomials of degree $4$, $3$ and $1$ respectively, and we consider the one-parameter family of quartics $\mathscr{X}\to \Delta$, where
\begin{equation}
   \mathscr{X}=\{sF_X+F_YF_P=0\}\subset \Delta\times \mathbb P^4, \label{quartic_family}
\end{equation}
and $\Delta$ a small disk centered at $0\in \mathbb C$. The special fiber is $Y\cup P$, and the general fiber $X_s$ is a smooth quartic 3-fold. 

According to Clemens \cite{Cle77,Cle82}, the topology of the family \eqref{quartic_family} is understood. In particular, \eqref{quartic_family} is dominated by a semistable family and there is a deformation retract of $X_s$ onto $Y\cup P$, which induces diffeomorphism of the complement $Y\setminus \mathcal{U}(Z)$ of a neighborhood of $Z=Y\cap P$ into a smooth submanifold $X_s'$ of $X_s$. Therefore, any 3-cycle in $Y\setminus \mathcal{U}(Z)$ can be deformed to the nearby quartics. 

So we want to find a tube class on $Y$ away from $Z$. To do this, we have to make sure both the primitive vanishing cycle $\delta$ and the tube class $\tau_g(\delta)$ are all disjoint from $Z$. By the diffeomorphism constructed above, such a 3-cycle flows to the nearby smooth quartics.

However, there are a few subtle questions: Why are the deformed classes on the nearby quartics still tube classes? Why are these classes nonzero? What about the singularities on the total space \eqref{quartic_family}? 

In the rest of this section, we will resolve these questions and facilitate the degeneration approach. In fact, it works more generally for hypersurfaces in arbitrary dimensions. We want to prove the following.

 \begin{theorem}\label{thm_degeneration}
 Suppose Conjecture \ref{Conj_Clemens} holds for smooth hypersurfaces of degree $d_0$ in $\mathbb P^{n+1}$, where $n$ is an odd number. Then it holds for smooth hypersurfaces of degree $d_0+1$ in $\mathbb P^{n+1}$.
   
 \end{theorem}

 The proof breaks up into several steps.

\subsection*{Step 1. Vanishing Cycles on the Hyperplane Complement.}
Let $Y\subseteq \mathbb P^{n+1}$ be a smooth hypersurface. Then one the vanishing homology of $N$ is kernel of the pushforward $H_{n}(Y,\Z)\to H_{n}(\mathbb P^{n+1})$. This agrees with \eqref{eqn_vanishinghomology} by a Veronese embedding, which sends $Y$ to a hyperplane section of the image of $\mathbb P^{n+1}$ in a larger projective space.

Let $H\subseteq \mathbb P^{n+1}$ as a general hyperplane, and let $Z=Y\cap H$. Then there is an exact sequence of the pair $(Y, Y-Z)$ 
\begin{equation}\label{eqn_seq}
    H_{n}(Y-Z,\Z)\xrightarrow{f} H_{n}(Y,\Z)\xrightarrow{g} H_{n-2}(Z,\Z).
\end{equation}

The map $g$ is the composite of $H_{n}(Y,\Z)\to H_{n}(Y,Z,\Z)\cong H_{n}(\mathcal{V}(Z),\mathcal{V}(Z)-Z,\Z)\cong H_{n-2}(Z,\Z)$, where the second map is excision, and the last map is Thom isomorphism. Homologically, $g$ sends a class $A\in H_{n}(Y,\Z)$ to the intersection $A\cap [H]$.

$\textup{Im}(f)$ consists of cycles on $Y$ that are homologous to cycles on the complement of a hyperplane section. Such a class is called \textit{finite cycle} in \cite[p.483]{Griffiths}. It turns out that vanishing cycles are precisely finite cycles.

\begin{lemma}(c.f. \cite[Prop. 7.3]{Griffiths})\label{lemma_vanishingaffine}
$H_{n}(Y,\Z)_{van}=\ker(g)=\textup{Im}(f)$.
\end{lemma}
\begin{proof}
We give a proof using cohomology. By Poincare-Lefschetz duality, the sequence \eqref{eqn_seq} is isomorphic to 
$$H^{n}(Y,Z,\Z)\to H^{n}(Y,\Z)\xrightarrow{r} H^{n}(Z,\Z).$$

The kernel of $r$ given by restriction is exactly the primitive cohomology $H^{n}(Y,\Z)_{prim}$, which coincides with the vanishing cohomology since $N$ is a hypersurface. (Note we also used this fact in the proof of Lemma \ref{lemma}). so its Poincare dual, $\ker (g)$ is the vanishing homology.
\end{proof}

\subsection*{Step 2. Tube Cycles on the Hyperplane Complement.}
Recall that in Corollary \ref{cor_uniformLefschetz}, we have an analytic open set $U$ is obtained by removing finitely many small closed disks from $\mathbb L$  centered at $\mathbb L\cap X_0^*$. Choose a base point $t_0\in U$ (in particular, we may choose $t_0=c$ in Proposition \ref{prop_c}).

Let $\mathcal{V}(Z)$ be a tubular neighborhood of $Z=Y\cap P$ in $Y$. By restricting to a hyperplane section, $\mathcal{V}(Z_t)=\mathcal{V}(Z)\cap H_t$ is a tubular neighborhood of $Z_t$ in $Y_t$.

\begin{proposition}
    The family of manifolds $\{Y_t-\mathcal{V}(Z_t)\}_{t\in U}$ is topologically locally trivial.
\end{proposition}
\begin{proof}
  Let $\mathcal{Y}\to \mathbb L$ the total space of the pencil whose fiber is the hyperplane section $Y_t$. It contains $\mathcal{Z}$ whose fiber is $Z_t$. Let $\pi:\mathcal{Y}_U\to U$ restriction of the fibration over $U$. Then $\pi$ is proper submersion and is locally trivial by the classical Ehresmann's theorem. 

On the other hand, $Z_t$ is smooth for $t\in U$ (c.f. Corollary \ref{cor_uniformLefschetz}), and $\mathcal{V}(Z_t)$ is a disk bundle over $Z_t$, so the boundary $\partial\overline{\mathcal{V}(Z_t)}$ is a circle bundle over $Z_t$. Now we remove $\cup_{t\in U}\mathcal{V}(Z_t)$ from $\mathcal{Y}_U$, and denote the resulting manifold as $M$. Then the restriction $\pi|_{\partial M}:\partial M\to U$ is submersion, and the fiber over $t$ is just the circle bundle over $Z_t$. By Lemma \ref{EhresmannWithBoundary}, $M\to U$ is locally trivial.
\end{proof}

Consequently, it makes sense to talk about the monodromy on the homology of $Y_t-Z_t$ over the base $U$.

However, the open manifold carries more homology: There is an exact sequence 
$$H_{n-2}(Z_{t_0},\mathbb Z)\to H_{n-1}(Y_{t_0}-Z_{t_0},\mathbb Z)\to H_{n-1}(Y_{t_0},\mathbb Z)\to 0.$$

The monodromy of $H_{n-1}(Y_{t_0},\Z)$ is a quotient of the monodromy group of $H_{n-1}(Y_{t_0}-Z_{t_0},\Z)$. The latter one also deposits on the monodromy on $Z_{t_0}$ — This happens when the loop goes around a point $w$ in the pencil where $Z_w$ has an ordinary node. (e.g., when $n=3$ and $d=3$, $Z_{t_0}$ is a cubic curve, and $Y_{t_0}$ is a cubic surface) So we need to carefully choose the representatives of $\pi_1(U,t_0)$ to avoid the monodromy on $H_{n-2}(Z_{t_0},\Z)$.

\begin{figure}[h]
\includegraphics[width=8cm]{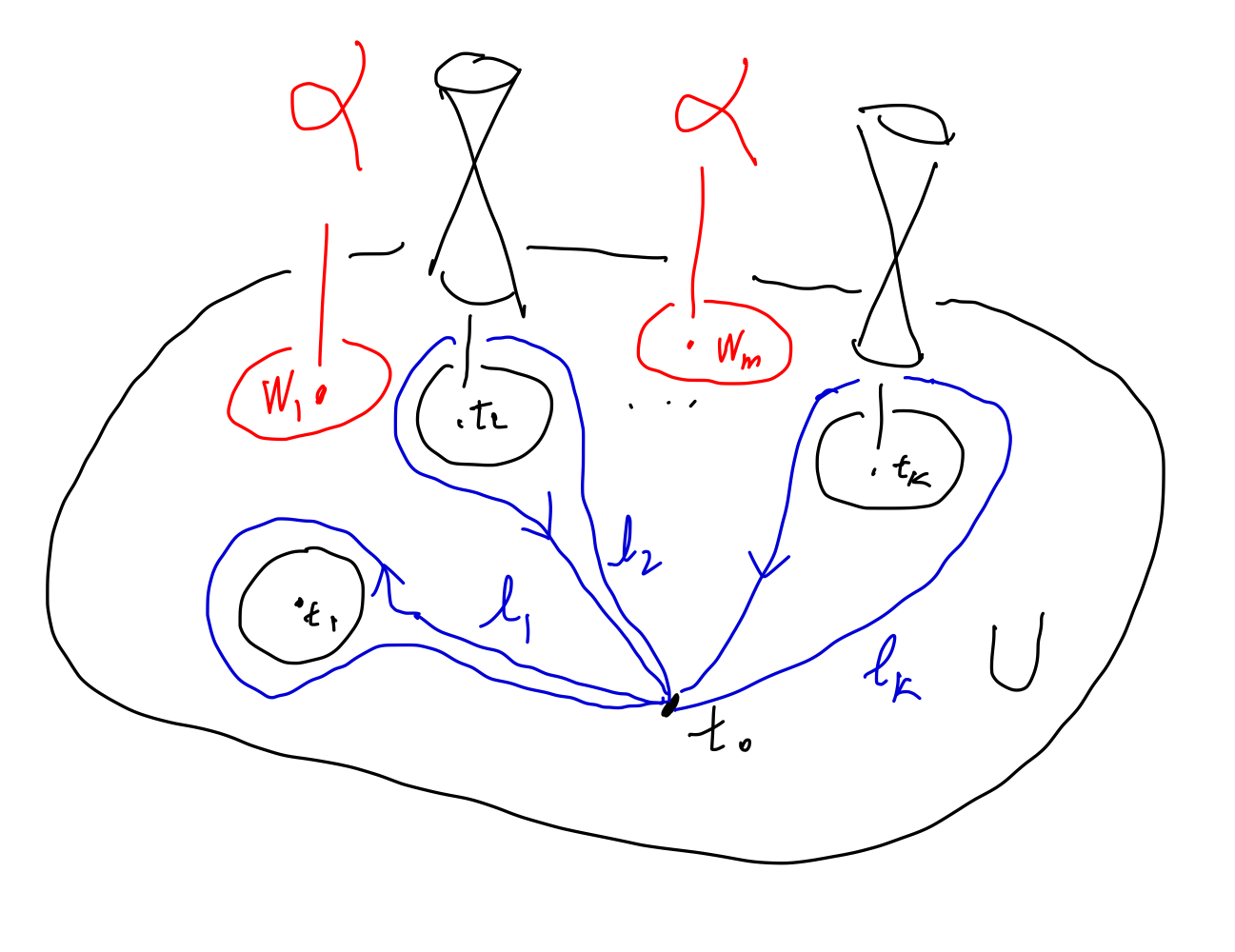}
\caption{Uniform Lefschetz Pencil}
\end{figure}

By construction in Corollary \ref{cor_uniformLefschetz}, the Lefschetz pencil $\mathbb L$ intersects both $Y^*$ and $Z^*$ transversely at distinct points. Denote them as $t_1,\ldots, t_k$ and $w_1,\ldots, w_m$. Denote $U_Y=\mathbb L\setminus \{t_1,\ldots, t_k\}$. Then the inclusion $U\hookrightarrow U_Y$ induces $\pi_1(U,t_0)\twoheadrightarrow \pi_1(U_Y,t_0)$. 
\begin{lemma}\label{lemma_Step2}
We choose the loop $l_i$ obtained from joining $t_0$ to $t_i$ going around and comes back to $t_0$, for $i=1,\ldots, k$, and does not include $w_j$. Then 
\begin{itemize}
    \item[(1)] the subgroup $G$ of $\pi_1(U,t_0)$ generated by the loops $l_1,\ldots, l_k$ is isomorphic to $\pi_1(U_Y,t_0)$ via the natural map, and
    \item[(2)] the subgroup induces trivial monodromy on $H_{n-2}(Z_{t_0},\Z)$.
\end{itemize}
\end{lemma}

Consequently, we can perform monodromy safely on the open manifold.
\begin{corollary}\label{cor_tubeOnOpen}
Let $\delta\in H_{n-1}(Y_{t_0},\mathbb Z)_{van}$ and $g\in \pi_1(U,t_0)$ and stabilizes $\delta$ via monodromy. Then the tube class $\tau_g(\delta)$ is represented by a $n$-cycle contained in $Y\setminus \mathcal{V}(Z)$. 

\end{corollary}
\begin{proof}
    By Lemma \ref{lemma_vanishingaffine}, we can choose $\delta$ to be supported on $Y_t-\mathcal{V}(Z_t)$. By Lemma \ref{lemma_Step2}, there are $l_1,...,l_k$ of $\pi_1(U_Y,t_0)$ which acts trivially on homology of $Z_{t_0}$. So if $g$ stabilizes in $Y_t$, then $g=l_{i_1}\cdots l_{i_p}$ stabilizes $\delta$ in $Y_t-\mathcal{V}(Z_t)$. So the trace of $\delta$ along the chosen loop forms a tube cycle $\tau_g(\delta)$ contained in $Y\setminus \mathcal{V}(Z)$.
\end{proof}

\subsection*{Step 3. Deformation of Tube Cycle to Nearby Smooth Fiber.}Now we are in the same situation as in Section \ref{sec_LefschetzInFamily}. Consider the total space of the family of hypersurfaces of degree $d$ over a small disk. Let $F_X$ be the defining equation of a degree $d$ hypersurface, and let $F_Y$ be the defining equation of a degree $d-1$ hypersurface. Let $F_P$ be the defining equation of a general hyperplane.
\begin{equation}\label{eqn_degneration}
    \mathcal{X}=\{(x,s)\in \mathbb P^{n+1}\times \Delta| (sF_X+F_YF_H)(x)=0\}\to \Delta,\ (x,s)\mapsto s.
\end{equation}

The total space is singular along the codimension 3 locus $S=\{s=0, F_X=0,F_Y=0,F_P=0\}$. As an example, when $n=3$ and $d=4$, $S$ is a curve of genus 19. $\mathcal{X}$ has transversal $A_1$ singularity along $S$. According to \cite{Cle77}, one needs to resolve the singularities to fully understand the topology and the limiting Hodge theory.

Just as in the case of the small resolution of $A_1$ singularity for a threefold, we resolve the singularity by blowing up $P$ defined by $F_H=s=0$ in the total space and get a small resolution 
$$\tilde{\mathcal{X}}\to \mathcal{X}.$$

Now the family $\tilde{\mathcal{X}}\to \Delta$ is semistable and the fiber at origin $\tilde{\mathcal{X}}_0$ becomes $Y\cup \tilde{P}$, where $\tilde{P}$ is the blow-up of $P$ along $S$, and the two components intersect transversely along $\tilde{Z}\cong Z$.  Then, we can apply Clemens' result.

\begin{lemma}(c.f. \cite[Theorem 5.7]{Cle77}, \cite[Lecture 1]{Cle82})\label{lemma_Clemens}
    There exist tubular neighborhoods $\mathcal{U}(Y)$ and $\mathcal{U}(\tilde{P})$ of $Y$ and $\tilde{P}$ in the total space $\mathcal{X}$, and projections $\pi_Y:\mathcal{U}(Y)\to Y$ and  $\pi_{\tilde{P}}:\mathcal{U}(\tilde{P})\to  \tilde{P}$ such that 

    \begin{itemize}
        \item[(1)] $\mathcal{U}(\tilde{Z}):=\mathcal{U}(Y)\cap \mathcal{U}(\tilde{P})$ is a tubular neighborhood of $\tilde{Z}$, and
        \item[(2)]    $\pi_Y\circ\pi_{\tilde{P}}=\pi_{\tilde{P}}\circ\pi_Y$ on $\mathcal{U}(\tilde{Z})$ and makes $\mathcal{U}(\tilde{Z})\to \tilde{Z}$ an analytic polydisk bundle.
        \item[(3)] There exist holomorphic coordinates $z_1,z_2,w_1,\ldots,w_{n-1}$ on $\mathcal{U}(\tilde{Z})$, such that $z_1$ is constant along $\pi_{\tilde{P}}$, $z_2$ is constant along $\pi_Y$, and
        \begin{equation}\label{eqn_crossing}
            z_1z_2=u\cdot s,
        \end{equation}
        where $u=u(w_1,\ldots,w_{n-1})$ is a non-vanishing analytic function depending on points of $\tilde{Z}$.
    \end{itemize}    
\end{lemma}
\begin{figure}[h]
\includegraphics[width=7cm]{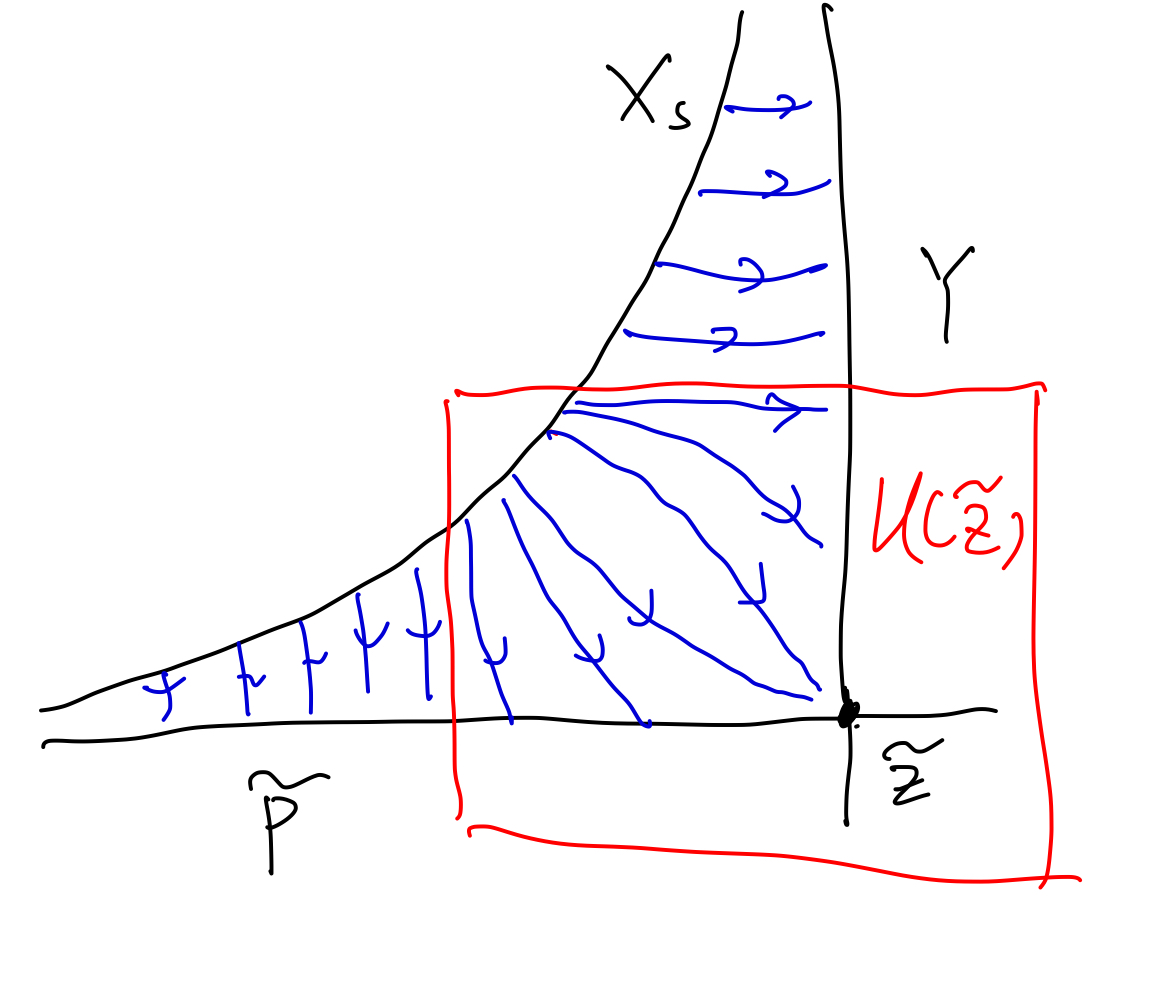}
\caption{Degeneration of Hypersurfaces}
\end{figure}

As a consequence, the boundary of $\mathcal{U}(\tilde{Z})$ has two components $|z_i|=\varepsilon$, for $i=1,2$ and they are both diffeomorphic to $S^1\times \Delta_j$-bundle over $\tilde{Z}$, where the disk $\Delta_j$ has coordinate $z_j$ with $j\neq i$. Then up to restricting $\Delta$ to a smaller disk, $\tilde{\mathcal{X}}\setminus \mathcal{U}(\tilde{Z})$ has two connected component $\tilde{\mathcal{X}}'$ and $\tilde{\mathcal{X}}''$, with $\tilde{\mathcal{X}}'$ contained in $\mathcal{U}(Y)$ and disjoint from $\tilde{P}$ (similarly, the other way around for $\tilde{\mathcal{X}}''$). 

\begin{claim}
    There is diffeomorphism 
    $$\partial\tilde{\mathcal{X}}'\cong \Delta\times \mathcal{S}$$
    that preserves projection to $\Delta$, where $\mathcal{S}$ is a $S^1$-bundle over $\tilde{Z}$.
\end{claim}
\begin{proof}
   Apply \eqref{eqn_crossing} restricted to $|z_2|=\varepsilon$, the diffeomorphism is given by $s=\frac{z_1z_2}{u}$.
\end{proof}

Now, as a consequence of Lemma \ref{EhresmannWithBoundary}, we have  

\begin{corollary}
 Up to restricting to a smaller disk centered at $s=0$, the family
\begin{equation} \label{eqn_TubeofY}
    \pi':\tilde{\mathcal{X}}'\to \Delta
\end{equation}
is $C^{\infty}$ trivial. In particular, we have a  family $\{X_s'\}_{s\in \Delta}$ of submanifolds of $X_s$ with boundaries. When $s=0$, $X_0'\subseteq Y$ is disjoint from $P$.
\end{corollary}
Recall that in Corollary \ref{cor_tubeOnOpen}, we already find the tube cycle $\tau_g(\delta)$ supported on the complement of a tubular neighborhood of $Z$. Then we may apply the diffeomorphism induced from \eqref{eqn_TubeofY} and deform the tube cycle $\tau_g(\delta)$ to the nearby $X_s$. However, it is not clear at this moment why such a class arises from tube mapping over a primitive vanishing cycle on hyperplane sections of $X_s$ — the trivialization of \eqref{eqn_TubeofY} may not preserve hyperplane sections.

To fix this issue, we need a parametric version of \eqref{eqn_TubeofY} preserving hyperplane sections. This amounts to considering the incidence manifold 
$$W=\{(x_s,t)\in \tilde{\mathcal{X}}'\times U| x_s\in X_s'\cap H_t\}.$$

There are two natural projection from $W$: $\sigma: (x_s,t)\mapsto x_s$, which blows up the base locus of Lefschetz pencil $\cup_{s\in\Delta}(X_s'\cap H_{t_1}\cap H_{t_2})$, and $q:(x_s,t)\mapsto (s,t)$. In particular, there is a commutative diagram
\begin{figure}[ht]
    \centering
\begin{tikzcd}
W\arrow[r,"\sigma"] \arrow[d,"q"]& \tilde{\mathcal{X}}' \arrow[d,"\pi'"] \\ 
\Delta\times U\arrow[r]& \Delta. 
\end{tikzcd}
\end{figure}

\begin{proposition}\label{prop_open2dimfamily}
   $q:W\to \Delta\times U$ is a locally trivial fiber bundle. Consequently, we get a two-parameter locally trivial family of hyperplane sections
    $$\{X_s'\cap H_t\}_{(s,t)\in \Delta\times U}.$$
\end{proposition}

\begin{proof}
$q$ is clearly a submersion in the interior of $W$. To show $q$ is submersion on the boundary $\partial W$, note the fiber of $\partial W\to \Delta\times U$ over $(s,t)$ is transversal intersection $\partial \tilde{\mathcal{X}}'\cap H_t$, which is diffeomorphic to a $S^1$-bundle over $Z_t=Z\cap H_t$. Since $Z_t$ is smooth for $t\in U$ by Corollary \ref{cor_uniformLefschetz}, the fiber of $q|_{\partial W}$ varies diffeomorphically, and $q|_{\partial W}$ is submersion. Now the result follows from Lemma \ref{EhresmannWithBoundary} again. \end{proof}

Proposition \ref{prop_open2dimfamily} allows us we perform monodromy in the one-dimensional analytic open space $U$ simultaneously: We can subdivide the loop $g\in \pi_1(U,t_0)$ into line segments $\cup_il_i$, such that each $l_i$ is contained in a contractible neighborhood $U_i$. We choose a point $c_i$ to be the tail of $l_i$. Then, up to restricting the $\Delta$ to a smaller disk centered at $s=0$, and for each $i$, there is a trivialization 
$$\psi_i: W|_{\Delta\times U_i}\xrightarrow{\cong} (X_{0}'\cap H_{c_i})\times \Delta\times U_i$$
that preserves projection to $\Delta\times U_i$.  Now, let $\alpha$ be a topological cycle on $X_0'\cap H_{c_i}$, then $\psi_i^{-1}(\alpha\times \{s\} \times l_i)$ spreads out the class over the segment $l_i$ for each $s\in \Delta$. Moreover, it concatenates to the chain $\psi_{i+1}^{-1}(\alpha\times \{s\} \times l_{i+1})$ since their endpoints differ by a boundary $\partial \Gamma$, where $\Gamma\subseteq X_{s}'\cap H_{c_{i+1}}$ is a chain induced by a family of self diffeomorphisms parameterized by $l_i$.

\subsection*{Proof of Theorem \ref{thm_degeneration}}

Now we are ready to give a proof of the main theorem of this section. Suppose Conjecture \ref{Conj_Clemens} holds for a hypersurface $Y$ odd degree $d_0=d-1$. We choose a one-parameter family of degree $d$ hypersurfaces with the special fiber union of a degree $d-1$ hypersurface and a hyperplane $Y\cup P$ as in \eqref{eqn_degneration}. Let $U$ be the one parameter family of hyperplanes as in Corollary \eqref{cor_uniformLefschetz}.

Now, let $\delta_0$ be a primitive vanishing cycle on the hyperplane section $Y\cap H_{t_0}$ and a loop $g\in \pi_1(U,t_0)$ such that $g_*(\delta_0)=\delta_0$. By Corollary \ref{cor_tubeOnOpen}, we can choose a representative such that the flat translates of $\delta_0$ along $g$ are contained in $Y_t\setminus \mathcal{U}(Z)$. Let $\delta_s$ in $X_s'\cap H_{t_0}$ be the translates of $\delta_0$. Then automatically $g$ stabilizes $\delta_s$ by Proposition \ref{prop_open2dimfamily}. Moreover, according to Proposition \ref{prop_c}, $\delta_s$ are still primitive vanishing cycles on $X_s\cap H_{t_0}$. The trace of $\delta_s$ along the same loop $g$ defines a topological cycle $A_s$ in the total space $W_s=\bigcup_{t\in U}(X_s'\cap H_t)$, and the image of $A_s$ in $H_n(X_s,\Z)$ via the composite $W_s\to X_s'\hookrightarrow X_s$ is the tube cycle $\tau_g(\delta_s)$. Moreover $\tau_{g}(\delta_s)$ is a $n$-cycle on $X_s$ and specializes to $\tau_{g}(\delta_0)$ as $s\to 0$. Finally, by Proposition \ref{prop_nonzero} below, $\tau_{g}(\delta_s)$ is nonzero in $H_n(X_s,\Z)$. Since $n$ is odd by assumption and $X_s$ is a hypersurface, $H_n(X_s,\Z)=H_n(X_s,\Z)_{prim}$. So $\tau_{g}(\delta_s)$ is nonzero in the primitive homology. So Conjecture \ref{Conj_Clemens} holds for hypersurfaces of degree $d=d_0+1$ as well.\qed

\begin{proposition}\label{prop_nonzero}
   Assume the tube class $\tau_g(\delta_0)$ is not the zero in $H_n(Y,\Z)$. Then the class $\tau_g(\delta_s)$ is nonzero in $H_n(X_s,\Z)$ for $s\neq 0$.
\end{proposition}
\begin{proof}

Recall in the beginning of Step 3, we produce a small resolution on the total space of the family \eqref{quartic_family} of degree $d$ hypersurfaces and get a semistable family
\begin{equation}
   \tilde{\mathcal{X}}\to \Delta. \label{semistable_family} 
\end{equation}

 There is an associated limiting mixed Hodge structure $H^n_{\textup{lim}}$ whose weight-$n$ filtration $W_nH^n_{\textup{lim}}$ is contributed from the image of $H^n(Y\cup \tilde{P})$, which fits into an exact sequence \cite[(4.6)]{Clemens_Degeneration}
\begin{equation}\label{eqn_ssfamily}
    0\to H^{n-1}(Y\cap \tilde{P},\mathbb Z)_{van}\to H^n(Y\cup \tilde{P},\mathbb Z)\to H^n(Y,\mathbb Z)_{prim}\oplus H^n(\tilde{P},\mathbb Z)_{prim}\to 0.
\end{equation}

 Alternatively, denote $\pi:\tilde{\mathcal{X}}^*\to \Delta^*$ the restriction of \eqref{semistable_family} to the punctured disk. Then $W_nH^n_{\textup{lim}}$ consists of invariant sections of the local system $j_*R^n\pi_{*}\underline{\mathbb Z}$, which are precisely $i^*j_*R^n\pi_{*}\underline{\mathbb Z}$.

Since by our construction, the tube cycle $\tau_{g}(\delta_s)$ specializes to $\tau_{g}(\delta_0)$ contained in $Y\setminus \mathcal{U}(P)$, and by assumption, $\tau_{g}(\delta_0)\in H_n(Y,\mathbb Z)_{prim}$ is nonzero. In particular, the Poincare dual of $\tau_{g}(\delta_0)$ is in $W_nH^n_{\textup{lim}}=i^*j_*R^n\pi_{*}\underline{\mathbb Z}$ from the exact sequence \eqref{eqn_ssfamily}. Therefore $\{\tau_{g}(\delta_s)\}_{s\in \Delta}$ defines a section $\eta$ in
$$j_*R^n\pi_{*}\underline{\mathbb Z}$$
over $\Delta$ with $\eta(0)\neq 0$. It follows that $\eta(s)$ is not zero for $s\neq 0$ close enough to $0$. In particular, for such $s$, $\tau_{g}(\delta_s)$ is not a zero class in $H_n(X_{s},\mathbb Z)$.   
\end{proof}

\subsection*{Proof of Theorem \ref{MainThm}}
We prove this by induction. The theorem is true for $d=3$ by Proposition \ref{prop_cubic3fold}. Then we inductively degenerate a hypersurface of degree $d$ in $\mathbb P^4$ to the union of a general hypersurface of degree $d-1$ and a hyperplane meeting transversely. Then, by Theorem \ref{thm_degeneration} and Lemma \ref{lemma}, if the theorem holds for degree $d-1$ hypersurface, it holds for degree $d$ hypersurface as well.
\qed

\section*{Appendix: Relative Ehresmann's Lemma}
Ehresmann's lemma says that proper submersion is topologically locally trivial. We want to show the same result holds for a pair.
\begin{lemma}\label{appendix} Let $N\subseteq M$ be a pair of smooth manifolds. Suppose $\pi:M\to B$ is a proper submersion to a smooth base, and moreover, the restriction $\pi|_{N}$ is also a submersion, then for each $b\in B$, there is a neighborhood $U$ of $b$ and a fiber preserving diffeomorphism
$$\Psi:\pi^{-1}(U)\xrightarrow{\cong}M_b\times U,$$
which satisfies 
$$\Psi(N\cap \pi^{-1}(U))=N_b\times U,$$
where $M_b=\pi^{-1}(b)$ and $N_b=\pi_{|N}^{-1}(b)$.
\end{lemma}

The theorem follows from Thom's second isotopy theorem \cite[Proposition 11.2]{Mather}. Here we present an elementary proof communicated to me by Gael Meigniez.

\begin{proof}
    First, apply Ehresmann's lemma to $M$, and by replacing $B$ with a smaller neighborhood, we can assume that $M=M_b\times B$ and that $\pi$ is the projection to the second coordinate. Then apply Ehresmann's lemma to $N$, we have a local trivialization
    $$\Psi_N:N\to N_b\times B$$
such that $\pi\circ \Psi_N=\pi_{|N}$. We can regard $\Psi_N$ as a parametric family of smooth embeddings
$$f_y:N\hookrightarrow M_b$$
parameterized by $y\in B$, namely
$$(f_y(x),y)=\Psi_N^{-1}(x,y).$$

According to the isotopy extension theorem \cite{Cerf, Palais}, a parametric family of embeddings of $N_b$ in $M_b$ extends to a parametric family of self-diffeomorphism of the ambient manifold $M_b$. In other words, there is a smooth family $F_y$ of self-diffeomorphisms of $M_b$ such that $f_y=F_y|_{N_b}$ and $F_b$ is the identity map. So the trivialization $\Psi$ is the inverse of 
$$(x,y)\mapsto (F_y(x),y),$$
with $x\in M_b$ and $y\in B$.
\end{proof}

\begin{lemma}(Ehresmann Fibration for Manifolds with Boundary)\label{EhresmannWithBoundary}
    Let $M$ be a smooth manifold with boundary. Suppose $f:M\to B$ is a proper map onto a smooth base, and assume that both $f|_{M^{\circ}}:M^{\circ}\to B$ and $f|_{\partial M}:\partial M\to B$ are submersions, then $f$ is a locally trivial fiber bundle.
\end{lemma}
\begin{proof}
    We use the "double" construction. Let $M'$ be the gluing of two copies of $M$ along the boundary. Then there is a natural map $f':M'\to B$, which is proper submersion, and the restriction on $N=\partial M$ is also submersion. Then, by Lemma \ref{appendix}, $f'$ and $f'|_{N}$ can be simultaneously locally trivialized. Then the results follow by restricting to $M$.
\end{proof}

\bibliographystyle{plain}
\bibliography{bibfile}

\end{document}